\documentclass[12pt,reqno]{amsart} 

\usepackage{fullpage} 
\usepackage{amssymb,amsfonts}
\usepackage[all,arc]{xy}
\usepackage{enumerate}
\usepackage{mathrsfs}
\usepackage{color}   
\usepackage{hyperref}
\hypersetup{
    colorlinks=true, 
    linktoc=all,     
    linkcolor=blue,  
}
\usepackage{amsthm}
\usepackage{amsmath}
\usepackage{graphicx}
\usepackage{wasysym}
\usepackage{pgf,tikz}
\usetikzlibrary{positioning}
\usepackage{ytableau}

\newcommand{\bc}{\mathbb C}

\newcommand{\bz}{\mathbb Z}

\newcommand{\ba}{\mathbb A}

\newcommand{\la}{\langle}
\newcommand{\ra}{\rangle}

\newcommand{\bs}{\backslash}

\newcommand{\al}{\alpha}

\newcommand{\lam}{\lambda}

\DeclareMathOperator{\GL}{GL}
\DeclareMathOperator{\Hom}{Hom}
\DeclareMathOperator{\Ind}{Ind}

\newcommand{\sco}{\mathcal{O}}

\newtheorem{Thm}{Theorem}[section]
\newtheorem{Prop}[Thm]{Proposition}
\newtheorem{Lem}[Thm]{Lemma}

\theoremstyle{definition}
\newtheorem{Def}[Thm]{Definition}

\theoremstyle{remark}
\newtheorem{Rem}[Thm]{Remark}
\newtheorem{Ex}[Thm]{Example}

\theoremstyle{definition}

\title{Fourier Coefficients for Degenerate Eisenstein Series and the Descending Decomposition}

\author{Yuanqing Cai}
\address{Department of Mathematics, Boston College, Chestnut Hill, MA 02467-3806}
\email{yuanqing.cai@bc.edu}
\date\today
\begin{document}

\begin{abstract}
We determine the unipotent orbits attached to degenerate Eisenstein series on general linear groups. This confirms a conjecture of David Ginzburg. This also shows that any unipotent orbit of general linear groups does occur as the unipotent orbit attached to a specific automorphic representation. The key ingredient is a root-theoretic result. To prove it, we introduce the notion of the descending decomposition, which expresses every Weyl group element as a product of simple reflections in a certain way. It is suitable for induction and allows us to translate the question into a combinatorial statement.
\end{abstract}

\maketitle


\section{Introduction}

Knowledge of Fourier coefficients is one of the most important tools in the theory of automorphic representations.  Such Fourier coefficients are parameterized by unipotent orbits. In this article, we study Fourier coefficients for degenerate Eisenstein series on general linear groups.

Let $F$ be a number field and $\ba$ be its adele ring. Let $G$ be a connected reductive group over $F$.
Let $\sco$ denote a unipotent orbit of the group $G$. As explained in Ginzburg \cite{ginzburg2006certain}, one can associate to this unipotent orbit a unipotent subgroup $U_2(\sco)$ of $G$, and a family of characters  of $U_2(\sco)(F)\bs U_2(\sco)(\ba)$. Let $\psi_{U_2(\sco)}$ denote such a character. With this data, one can define a Fourier coefficient of an automorphic form $f$ as the integral
\begin{equation}\label{eq:1}
\int\limits_{U_2(\sco)(F)\bs U_2(\sco)(\ba)} f(ug)\ \psi_{U_2(\sco)}(u)\ du.
\end{equation}

The unipotent orbits are a partially ordered set, and for classical groups they are identified with partitions, based on the Jordan decomposition. Given an automorphic representation $\pi$ on $G(\ba)$, we say that a unipotent orbit $\sco$ is attached to $\pi$ if all Fourier coefficients of the form Eq. (\ref{eq:1}) for larger or incomparable orbits vanish identically and some coefficient for $\sco$ is nonzero.
 The determination of unipotent orbits attached to certain residue representations has played an important role in the descent method (see Ginzburg, Rallis and Soudry \cite{descentbook}). According to Ginzburg's dimension equation formalism on Rankin-Selberg integrals (Ginzburg \cite{ginzburg2006eulerian,ginzburgaim}), this information is also frequently useful in the construction of Eulerian global integrals. The unipotent orbits attached to an automorphic representation also have a connection to its Arthur parameter (see Jiang \cite{jiang2012automorphic}).



In this paper we determine the unipotent orbits attached to degenerate Eisenstein series on general linear groups. From now on, let $G=\GL_r$. Let $\mu=(t_1\cdots t_b)$ be a partition of $r$  with $t_1\geq \cdots \geq t_b>0$. Let $\mu^\top=(r_1\cdots r_a)$ denote the transpose of $\mu$. Let $P=P_{\mu^\top}$ be the standard parabolic subgroup whose Levi subgroup is $M_{\mu^\top}\cong\GL_{r_1}\times \cdots \times \GL_{r_a}$.
 Let $\delta_P$ be the modular quasicharacter of $G$ with respect to $P$. We denote by $E_\mu(g,\underline{s})$ the Eisenstein series which corresponds to the induced representation $\Ind^{G(\ba)}_{P(\ba)}\delta_{P}^{\underline{s}}$. Here $\underline{s}=(s_1,\cdots, s_{a})$ denotes a multi-complex variable.

\begin{Thm}[Theorem \ref{thm:main-thm}]\label{thm:1.1}
With the above notation, suppose that $\mathrm{Re}(s_i)\gg0$. Then the unipotent orbit attached to $E_\mu(g,\underline{s})$ is $\mu$.
\end{Thm}

This confirms Conjecture 5.1 in Ginzburg \cite{ginzburg2006certain}. A local analogue is given in Theorem \ref{thm:main-local}. This also implies that
any unipotent orbit of general linear groups can occur as the unipotent orbit attached to a specific automorphic representation.

The difficulty in the proof of Theorem \ref{thm:1.1} is to show  the vanishing part for incomparable orbits, and the obstruction is of a combinatorial nature. When $\mu$ is of the form $(a^b)$, the situation is easier and a similar argument for symplectic groups can be found in Jiang and Liu \cite{jiangliu2013arthur} Lemma 3.1. Indeed, an orbit $(p_1\cdots p_m)$ is greater than or not comparable to $(a^b)$ if and only if $p_1>a$. In other words, $(p_1\cdots p_m)$ is greater than $(a^b)$ in the lexicographical order. To generalize, we need to handle these incomparable orbits uniformly.

The key ingredient in our proof is a new root-theoretic result, which is connected to Theorem \ref{thm:1.1} via the  study of semi-Whittaker coefficients. Let $\lambda=(p_1 \cdots p_k)$ be a partition of $r$ (here we do not require $p_1\geq \cdots \geq p_k$). Let $P_\lambda$ be the standard parabolic subgroup of $\GL_r$ whose Levi subgroup $M_\lam\cong \GL_{p_1}\times \cdots\times \GL_{p_k}$. Fix a nontrivial additive character $\psi:F\bs\ba\to \bc^\times$. Let $\psi_\lambda:U(F)\bs U(\ba)\to \bc^\times$ be the  character  such that it acts as $\psi$ on the simple positive root subgroups contained in $M_\lam$, and acts trivially otherwise. With this data,  a $\lambda$-semi-Whittaker coefficient of an automorphic form $f$ is
\[
\int\limits_{U(F)\bs U(\ba)}f(ug)\psi_\lambda(u) \ du.
\]

There is a strong relation between these two types of Fourier coefficients. In the local context, this is developed in Moeglin and Waldspurger \cite{MoeW}, and  Gomez, Gourevitch and Sahi \cite{gomez2015generalized}. We prove a global version in Proposition \ref{prop:app-b-1}. Thus, to prove Theorem \ref{thm:1.1}, it suffices to prove the following result.

\begin{Thm}[Theorem \ref{thm:semi-whittaker}]\label{thm:1.2}
Suppose $\mathrm{Re}(s_i)\gg0$.
\begin{enumerate}[\normalfont(1)]
\item If there is an index $l$ such that $p_1+\cdots +p_l> t_1+\cdots + t_l$, then
\[
\int\limits_{U(F)\bs U(\ba)} E_\mu(ug,\underline{s})\psi_\lambda(u) \ du= 0
\]
for all choices of data.
\item The semi-Whittaker coefficient
\[
\int\limits_{U(F)\bs U(\ba)} E_\mu(ug,\underline{s})\psi_\mu(u) \ du
\]
is nonzero for some choice of data.
\end{enumerate}
\end{Thm}

When $\mathrm{Re}(s_i)\gg0$, by a standard unfolding argument, Theorem \ref{thm:1.2} is quickly reduced to the following root-theoretic result.

\begin{Thm}[Theorem \ref{thm:root-theoretic}]\label{thm:1-3}
Let $\Delta$ denote the set of  simple roots with respect to the standard Borel subgroup $B$. Let $\Phi^+$ and $\Phi^-$ be the set of positive roots and negative roots, respectively. Let $\Delta_\lambda$ denote the set of simple roots contained in $M_\lambda$. Let $\Phi_{\mu^\top}^-$ denote the set of negative roots contained in $M_{\mu^\top}$. Let $W(P_{\mu^\top})$ and $W(G)$ be the Weyl groups of $P_{\mu^\top}$ and $G$, respectively. Let $[W(P_{\mu^\top})\bs W(G)]$ be the set of minimal representatives for $W(P_{\mu^\top})\bs W(G)$.
\begin{enumerate}[\normalfont(1)]
\item If there is an index $l$ such that
        \[
            p_1+\cdots +p_l> t_1+\cdots+t_l,
            \]
            then for any $w\in [W(P_{\mu^\top})\bs W(G)]$, there exists $\al\in \Delta_\lam$ such that $w(\al)>0$.
   \item If $\lam=\mu$, then there exists a unique $w_\mu\in [W(P_{\mu^\top})\bs W(G)]$ such that $w_\mu(\al)\in \Phi^--\Phi^-_{\mu^\top}$ for all $\al\in \Delta_\mu$; and $w(\Delta_\mu)\cap \Phi^+\neq \emptyset$ for all $w\neq w_\mu$.
    \end{enumerate}
\end{Thm}

This result is analogous to Casselman and Shalika \cite{casselman1980unramified} Lemma 1.5. We prove it by translating the root-theoretic problem into a combinatorial problem. Thus we need to analyze the action of the Weyl group on the simple roots with care.   To do so, we introduce the notion of descending decomposition, which is a systematic way to write down all the Weyl group elements. The descending decomposition expresses every element as a product of simple reflections in a certain way and it is suitable for induction. To simplify notations we let $G=\GL_{r+1}$ and $W(G)$ be its Weyl group. Choose the following long word decomposition
\[
w_0=s_1(s_{2} s_{1})(s_{3} s_{2} s_{1}) \cdots (s_r s_{r-1}\cdots s_1),
\]
where $s_1,\cdots,s_r$ are simple reflections in $W(G)$ with the usual labeling.
In fact, the expression $s_is_{i-1}\cdots s_1$ is the long word in the set of minimal representatives for $W(\GL_{i})\bs W(\GL_{i+1})$. Choosing a string starting with $s_i$ in each $(s_is_{i-1}\cdots s_1)$, and multiplying them,  we obtain an element in $W(G)$. Surprisingly, every $w\in W(G)$ has a unique expression of this form.

To be more precise,
let
\[
D=\{(k_1,\cdots, k_r): k_i\in \bz, \ 1\leq k_i\leq i+1\}.
\]
For each $k_i$, define
\begin{equation}
\pi_{k_i}=\left\{
\begin{aligned}
&s_{i}s_{i-1}\cdots s_{k_i} &\text{ if }1\leq k_i\leq i,\\
&e \text{ (the identity element)} &\text{ if }k_i=i+1.\\
\end{aligned}
\right.
\end{equation}
Define a map
\[
\pi: D\to W(G), \qquad (k_1,\cdots, k_r)\mapsto \pi_{k_1}\cdots \pi_{k_r}.
\]

\begin{Prop}(Proposition \ref{prop:descending-decomposition})
The map
$\pi$
 is a bijection. Moreover, $\pi(k_1,\cdots,k_r)$ is reduced.
\end{Prop}

 This leads to a natural way to express elements in $[W(P_{\mu^\top})\bs W(G)]$ as products of simple reflections (Theorem \ref{thm:coset-representative}). It allows us to compute the action of the Weyl group on the set of simple roots systematically, and  translate Theorem \ref{thm:1-3} into a combinatorial fact.

Finally we remark that for  other Cartan types, the situation is more complicated. For classical groups, it is known that only special unipotent orbits can occur as the unipotent orbits attached to automorphic representations (see \cite{ginzburg2006certain} Theorem 3.1, \cite{jiang2014raising}) and further assumptions are required for cuspidal representations (see \cite{ginzburg2006certain} Section 4, \cite{ginzburg2003fourier} Theorem 2.7). The interested reader is also referred to \cite{ginzburg2006certain} Section 5.2 for further information.




The rest of this paper is organized as follows. Section \ref{sec:descending}  introduces the descending decomposition and constructs a set of coset representatives.  Section \ref{sec:root-theoretic} is devoted to proving the root-theoretic result.    Section \ref{sec:degenerate-eis-series} introduces degenerate Eisenstein series and proves Theorem \ref{thm:semi-whittaker}. We introduce Fourier coefficients associated with unipotent orbits in Section \ref{subsec:fourier-unipontent}, and establish the connection between these two types of Fourier coefficients in Section \ref{sec:relations-between-coefficients}. The main result, i.e. the determination of the unipotent orbits attached to degenerate Eisenstein series, is given in Theorem \ref{thm:main-thm}.

\subsection*{Acknowledgement}
The descending decomposition was first discovered by Mark Reeder in a different context.
The author would like to thank him for many helpful discussions. The author would also like to thank Solomon Friedberg and David Ginzburg for the encouragement and their helpful comments. This work was supported by the National Science Foundation, grant number 1500977.

\section{Descending Decomposition and Coset Representatives} \label{sec:descending}

Let $G=\GL_{r+1}$. Let $B$ denote the standard Borel subgroup of $G$ with torus $T$ and unipotent radical $U$. Let $\Phi$ be the root system of $G$. Let $\Delta=\{\al_1,\cdots,\al_r\}$ denote the set of  simple roots with respect to $B$ (with the usual labeling). Let $\Phi^+$ and $\Phi^-$ be the set of positive roots and negative roots, respectively.
For each simple root $\al_i\in \Delta$, let $s_i$ denote the corresponding simple reflection. The Weyl group $W(G)$ of $G$ is generated by $s_1, \cdots, s_r$ and $W(G)\cong S_{r+1}$.

\subsection{Descending Decomposition}
Choose the following \textit{nice} long word decomposition (in the sense of Littelmann \cite{littelmann1998cones} Section 4)
\[
w_0=s_1(s_{2} s_{1})(s_{3} s_{2} s_{1}) \cdots (s_r s_{r-1}\cdots s_1).
\]
Let
\[
D=\{(k_1,\cdots, k_r): k_i\in \bz, 1\leq k_i\leq i+1\}.
\]
For each $k_i$, define
\begin{equation}\label{eq:pi-k}
\pi_{k_i}=\left\{
\begin{aligned}
&s_{i}s_{i-1}\cdots s_{k_i} &\text{ if }1\leq k_i\leq i,\\
&e \text{ (the identity element)} &\text{ if }k_i=i+1.\\
\end{aligned}
\right.
\end{equation}
Each $\pi_{k_i}$ is called a \textit{cycle} (notice this is slightly stronger than the usual ``cycle'' in $S_{r+1}$). For convenience, sometimes we use $s_{i_1 i_2\cdots i_N}$ to denote $s_{i_1}s_{i_2}\cdots s_{i_N}$ in each cycle. Define a map
\[
\pi: D\to W(G), \qquad (k_1,\cdots, k_r)\mapsto \pi_{k_1}\cdots \pi_{k_r}.
\]

\begin{Ex}
If $r=2$, then the Weyl group is $S_3$. The long word decomposition we consider is
\[
w_0=(s_1)(s_2 s_1).
\]
It is easy to check:
\[
\pi(2,3)=e, \qquad \pi(2,2)=s_2, \qquad  \pi(2,1)=s_{21},
\]
\[
\pi(1,3)=s_1, \qquad  \pi(1,2)=s_1s_2, \qquad  \pi(1,1)=s_1s_{21}.
\]
These are all the elements in $S_3$.
\end{Ex}

\begin{Prop}\label{prop:descending-decomposition}
The map
$\pi: D\to W(G)$
 is a bijection. Moreover, $\pi(k_1,\cdots,k_r)$ is reduced.
\end{Prop}

\begin{proof}
We first prove that the map $\pi$ is surjective.  The case $r=1$ is clear. We assume the result is true for $r-1$, and prove it for $r$.  Any Weyl group element $w$ can be written as a product of simple reflections. Choose a reduced expression for $w$ such that the number of $s_r$  is a minimum. If the expression does not contain $s_r$, then we are done by induction. If it contains $s_r$, we  show that it contains at most one $s_r$. Suppose there are two.  By applying induction on the expression between the two $s_r$'s,   $w$ can be written in one of  the following forms:
 \[
 \cdots s_r \cdots s_{r-1}\cdots s_r \cdots, \text{ or }\cdots s_r \cdots s_r \cdots
 \]
 Notice that $s_r$ commutes with all the simple reflections except $s_{r-1}$. If there is no $s_{r-1}$, then the two $s_r$'s cancel with each other; if there is only one $s_{r-1}$, then we can use the relation  $s_r s_{r-1} s_r=s_{r-1} s_r s_{r-1}$ to reduce the number of $s_r$. Thus there is at most one $s_r$ in the reduced expression.

Now applying induction to the expression on the right-hand side of $s_r$, we obtain a product of cycles. Then in the reduced expression we can move every cycle on the right-hand side of $s_r$ to the left-hand side except the last cycle $(s_{r-1}s_{r-2}\cdots)$. Then apply induction again on the left-hand side to obtain the desired expression. This shows that the map $\pi$ is surjective.

Notice that the cardinalities of $D$ and $W(G)$ are both $(r+1)!$. Thus $\pi$ is a bijection.

Now we show that $\pi(k_1,\cdots,k_r)$ is a reduced expression. Let
$\Phi^-(w)=\{\al>0: w\al <0\}.$
If $\al$ is a simple root, then
\[
\ell ( w s_\al)
= \left\{
        \begin{aligned}
        \ell(w)+1, \qquad &\text{ if }\al\notin \Phi^-(w),\\
        \ell(w)-1, \qquad &\text{ if }\al \in \Phi^-(w).
        \end{aligned}
        \right.
\]
Thus, it suffices to show that
\[
\pi_{k_1}\cdots \pi_{k_{t-1}}(\al_t+\cdots +\al_j)>0
\]
for any $k_1,\cdots, k_{t-1}$ and all $j$,  $1\leq j\leq t$.

Notice that $\pi_{k_1}\cdots \pi_{k_{t-1}}\in \la s_1,\cdots, s_{t-1}\ra$. The coefficient of $\al_t$ in $\pi_{k_1}\cdots \pi_{k_{t-1}}(\al_t+\cdots +\al_j)$ must be $1$. This shows that $\pi_{k_1}\cdots \pi_{k_{t-1}}(\al_t+\cdots +\al_j)$ is a positive root.
\end{proof}

\begin{Rem}
If we identify $W(G)$ with the permutation group on $r+1$ elements such that $s_i$ corresponds to the transposition $(i,i+1)$, then
\[
\pi_{k_i}=s_i s_{i-1}\cdots s_{k_i} \longleftrightarrow (i+1, i, \cdots, k_i)\in S_{r+1}.
\]
This explains the title \textit{descending decomposition}; see also Reeder \cite{marknotes} Section 3.2.1 for an explicit construction of the inverse map and Section 7.1 for the natural appearance of the descending decomposition in the Bruhat decomposition of $\GL_{r+1}$ for the upper triangular Borel subgroup.
\end{Rem}

\subsection{Coset Representatives}

Let $(r_1\cdots r_a)$ be a general partition of $r+1$, i.e. we do not require $r_1\geq \cdots\geq r_a.$ Let $P$ denote the standard parabolic subgroup of $G$ whose Levi subgroup is $\GL_{r_1}\times \cdots \times  \GL_{r_a}$ embedded in $G$ via
\[
\GL_{r_1}\times \cdots \times \GL_{r_a}\to G: (g_1, \cdots, g_a)\mapsto \mathrm{diag}(g_1,\cdots, g_a).
\]

We use \textit{Young tableaux} to describe a nice set of coset representatives for $W(P)\bs W(G)$. For the partition $(r_1\cdots r_a)$, we require its Young diagram to have  $r_i$ boxes in the $i$th column. Notice that this is not the usual definition and we do \textit{not} require  $r_1\geq \cdots \geq r_a$.  (The definition in Section \ref{sec:root-theoretic} is different but they are consistent.) We fill in the boxes in the Young diagram with $k_0, k_1, \cdots, k_{r}$ from top to bottom in columns,  from the leftmost column to the rightmost column.  For example,
if the partition is $(342)$, then the Young tableau is
\[
\begin{ytableau}
k_0 & k_3 & k_7 \\
k_1 & k_4  & k_8 \\
k_2 & k_5\\
\none & k_6\\
\end{ytableau}.
\]
Notice that if we delete the first row, then the simple reflections with the remaining indices generate $W(P)$.

Now delete the entries in the first column. Define $D_{(r_1\cdots r_a)}$ to be the set of  such Young tableaux  with the following two conditions:
\begin{equation}\label{eq:condition-thm-2.4}
\begin{aligned}
&\text{(a) } k_i\in \bz \text{ and } 1\leq k_i\leq i+1.&\\
&\text{(b) }\text{The entries in each column are strictly increasing.}&
\end{aligned}
\end{equation}
For each $i$, define $\pi_{k_i}$ by Eq. (\ref{eq:pi-k}). Then condition (b) means that the lengths of $\pi_{k_i}$ in each column are non-increasing. For convenience, we usually write $\Pi_j$ for the product of $\pi_{k_i}$'s in the $j$th column, although this depends on the choices of $k_i$'s.

Given a Young tableau $Y\in D_{(r_1\cdots r_a)}$, by taking the product of $\pi_{k_i}$ in all the columns, we can define a Weyl group element $\Pi_2\cdots \Pi_a\in W(G)$. This defines a map $\pi_{(r_1\cdots r_a)}:D_{(r_1\cdots r_a)}\to W(P)\bs W(G)$.

\begin{Thm}\label{thm:coset-representative}
The map $\pi_{(r_1\cdots r_a)}$ is a bijection.
\end{Thm}

\begin{Ex}
If the partition is $(32)$, then the Young tableau is
\[
  \ytableaushort[k_]{03,14,2}.
\]
Deleting the entries in first column gives
\[
\begin{ytableau}
{ } & k_3  \\
{ } & k_4   \\
{ } &\none\\
\end{ytableau}.
\]
The assumptions are
\[
1\leq k_3\leq 4, \qquad 1\leq k_4\leq 5,\qquad k_3<k_4.
\]
All the elements of the form $\pi_{k_3}\pi_{k_4}$ give a set of coset representatives for the quotient $W(\GL_3\times \GL_2)\bs W(\GL_5)$. There are $10$ elements:
\begin{alignat*}{5}
&e&& && && \\
&s_3, && s_3s_4,&& && \\
&s_{32},   && s_{32}s_{4}, &&  s_{32}s_{43},&& \\
&s_{321},  && s_{321}s_{4}, &&  s_{321}s_{43}, && \ s_{321}s_{432}.\\
\end{alignat*}

\end{Ex}

\begin{Ex}
If the partition is $(323)$, then the Young tableau is
\[
\begin{ytableau}
k_0 & k_3 & k_5 \\
k_1 & k_4  & k_6 \\
k_2 &\none& k_7\\
\end{ytableau}.
\]
Deleting the entries in the first column gives
\[
\begin{ytableau}
{ }& k_3 & k_5 \\
{ } & k_4  & k_6 \\
{ } &\none& k_7\\
\end{ytableau}.
\]
Then the assumptions are
\[
1\leq k_3\leq 4, \qquad 1\leq k_4\leq 5, \qquad k_3<k_4,
\]
\[
1\leq k_5\leq 6, \qquad 1\leq k_6\leq 7, \qquad 1\leq k_7\leq 8, \qquad k_5<k_6<k_7.
\]
All the elements of the form $(\pi_{k_3}\pi_{k_4})(\pi_{k_5}\pi_{k_6}\pi_{k_7})$ where the lengths of $\pi_{k_3}$ and $\pi_{k_4}$ (resp. $\pi_{k_5}$, $\pi_{k_6}$ and $\pi_{k_7}$) are non-increasing give a set of coset representatives for $W(P)\bs W(G)$.
\end{Ex}

We need the following lemma for the proof of Theorem \ref{thm:coset-representative}.
\begin{Lem}\label{lem:coset-representative}
If $i\geq j$, then
\[
(s_i \cdots s_j)(s_{i+1}\cdots s_j)=s_{i+1}(s_i\cdots s_j)(s_{i+1}\cdots s_{j+1}).
\]
Moreover, if $i\geq j\geq k$, then
\[
(s_i \cdots s_j)(s_{i+1}\cdots s_k)=s_{i+1}(s_i\cdots s_k)(s_{i+1}\cdots s_{j+1}).
\]
\end{Lem}

\begin{proof}
The second statement follows immediately from the first statement.
We prove the first statement by induction on $i-j$. When $i-j=0$, the result is clear since $(s_i)(s_{i+1}s_i)=s_{i+1}(s_i)(s_{i+1})$. For general $i,j$, notice that
\[
\begin{aligned}
&(s_i \cdots s_j)(s_{i+1}\cdots s_j)\\
=&(s_i\cdots s_{j+1})s_j (s_{i+1}\cdots s_{j+2})s_{j+1}s_j\\
=&(s_i\cdots s_{j+1}) (s_{i+1}\cdots s_{j+2})s_js_{j+1}s_j\\
=&(s_i\cdots s_{j+1}) (s_{i+1}\cdots s_{j+2})s_{j+1}s_js_{j+1}\\
=&s_{i+1}(s_i\cdots s_{j+1}) (s_{i+1}\cdots s_{j+2})s_js_{j+1}\text{ (by induction)}\\
=&s_{i+1}(s_i\cdots s_{j+1}s_j) (s_{i+1}\cdots s_{j+2}s_{j+1}).\\
\end{aligned}
\]
This finishes the proof.
\end{proof}

\begin{proof}[Proof of Theorem \ref{thm:coset-representative}]
Let $w\in W(P)\bs W(G)$. We need to choose a coset representative satisfying conditions (a) and (b) in Eq. (\ref{eq:condition-thm-2.4}). By Proposition \ref{prop:descending-decomposition}, we can choose a coset representative satisfying condition (a). We only need to show (b). If $\pi_{k_i}$ and $\pi_{k_{i+1}}$ are in the same row of the Young tableau, and $k_i\geq k_{i+1}$, then by Lemma \ref{lem:coset-representative},
\[
\pi_{k_i}\pi_{k_{i+1}}=s_{i+1}\pi_{k'_i}\pi_{k'_{i+1}}
\]
where $k'_i=k_{i+1}, k'_{i+1}=k_i+1$. Notice that $s_{i+1}\in W(P)$ and $k'_{i+1}> k'_i$. Therefore, we can always choose a coset representative satisfying both conditions.

To show that map is bijective, we again count the cardinalities of both sides. This is a combinatorial exercise and the proof is left to the reader.
\end{proof}

By Casselman \cite{casselman1995introduction} Lemma 1.1.2, there is a unique element of minimal length in each coset of $W(P)\bs W(G)$. The coset representative $\Pi_2\cdots \Pi_a$ is indeed this unique element in its coset. For completeness, we give a proof here.

\begin{Prop}\label{prop:minimal}
The element $\Pi_2\cdots \Pi_a$ is of minimal length in its coset in $W(P)\bs W(G)$.
\end{Prop}

\begin{proof}

We show that the coset representatives in Theorem \ref{thm:coset-representative} are of minimal lengths in each coset. We only have to show that for each $w$, $w^{-1}(\Delta_\lam) >0$ (\cite{casselman1995introduction} Section 1).
Equivalently, we need to show  that, if $\al_l\in \Delta_\lam$ (in other words, $l\neq r_1,r_1+r_2,\cdots, r_1+\cdots+r_{k-1}$), then $\al_l\notin \Phi^-(w^{-1})$.

Recall that if $w=s_{i_1}\cdots s_{i_N}$ is a reduced decomposition, then
\[
\{\al>0: w\al <0\}=\{\al_{i_N},s_{i_N}\al_{i_{N-1}},\cdots, s_{i_N}s_{i_{N-1}}\cdots s_{i_2}\al_{i_1} \}.
\]

Write $w=\pi_{k_{r_1}}\cdots \pi_{k_r}\in W(P)\bs W(G)$.  Thus it suffices to show that, given $\al_l\in \Delta_\lam$,
\begin{equation}\label{eq:app-minimal}
\pi_{k_{r_1}}\cdots \pi_{k_{j+1}}(\al_{j}+\cdots +\al_k)\neq \al_l
\end{equation}
for all $r_1\leq j\leq r$ and all possible $k$. Indeed, if $j<l$, then $\al_l$ cannot appear on the left-hand side. If $j>l$, then the coefficient of $\al_j$ in the left-hand side is $1$. In these two cases, Eq. (\ref{eq:app-minimal}) is  true. If $j=l$, then $k_{l-1}<k_l$ (since $\al_l\in \Delta_\lam$ and $k_l$ is not on the top row in the Young tableau). Therefore, for $k\geq k_l$,
\[
(s_{l-1}\cdots s_{k_{l-1}})(\al_{l}+\cdots +\al_k)=\al_l+\al_{l-1}+\cdots+\al_{k}+\al_{k-1}.
\]
The expression before $(s_{l-1}\cdots s_{k_{l-1}})$ is in $\la s_1,\cdots, s_{l-2} \ra$. Thus the coefficients of $\al_{l-1},\al_l$ in $\pi_{k_{r_1}}\cdots \pi_{k_{j+1}}(\al_{j}+\cdots +\al_k)$ are both $1$. Thus Eq. (\ref{eq:app-minimal}) also holds. This completes the proof.
\end{proof}

\section{Root-theoretic Results}\label{sec:root-theoretic}

The goal of this section is to prove a root-theoretic result, which is used in the proof of Theorem \ref{thm:semi-whittaker}.

Fix a partition $\mu=(t_1\cdots t_b)$ of $r+1$ such that $t_1\geq \cdots \geq t_b>0$. Let $\mu^\top=(r_1\cdots r_a)$ be its transpose. Let $P_{\mu^\top}$ be the parabolic subgroup whose Levi subgroup $M_{\mu^\top}\cong \GL_{r_1}\times \cdots \times \GL_{r_a}$.  We represent $w\in W(P_{\mu^\top})\bs W(G)$ by using the coset representatives in Theorem \ref{thm:coset-representative}.

 Let $\lam=(p_1 \cdots p_m)$ be a general partition of $r+1$. Let $\lam^\top=(q_1\cdots q_n)$ be its transpose.  Here $q_i=\#\{j:p_j\geq i\}$. In particular, $q_1=m$.  Let $P_{\lambda}$ denote the standard parabolic subgroup of $G$ whose Levi subgroup  is $M_\lambda\cong \GL_{p_1}\times \cdots \times  \GL_{p_m}$.

Now we are ready to state the main result in this section.

\begin{Thm}\label{thm:root-theoretic}
Let $\Delta_\lambda$ denote the set of simple roots contained in $M_\lambda$. Let $\Phi_{\mu^\top}^-$ denote the set of negative roots contained in $M_{\mu^\top}$.
\begin{enumerate}[\normalfont(1)]
\item If there is an index $l$ such that
        \[
            p_1+\cdots +p_l> t_1+\cdots+t_l,
            \]
            then for any $w\in W(P_{\mu^\top})\bs W(G)$, there exists $\al\in \Delta_\lam$ such that $w(\al)>0$.
    \item If $\lam=\mu$, then there exists a unique $w_\mu\in W(P_{\mu^\top})\bs W(G)$ such that $w_\mu(\al)\in \Phi^--\Phi^-_{\mu^\top}$ for all $\al\in \Delta_\mu$; and $w(\Delta_\mu)\cap \Phi^+\neq \emptyset$ for all $w\neq w_\mu$.
    \end{enumerate}
\end{Thm}

\subsection{Transpose of partitions}
The relation between $\lam$ and $\lam^\top$ can be easily read from the Young diagram associated with $\lam$. From now on, the Young diagram associated with  $\lam$ is the one that has $m$ rows, and $p_i$ boxes in the $i$th row. Here the definition also applies to general partitions, i.e. we do not assume the row sizes are weakly decreasing.  We also label the rows and columns with the entries they represent.  Here are examples with $\lam=(3221)$, $\lam^\top=(431)$ and $\lam'=(313)$, $\lam'^\top=(322)$:
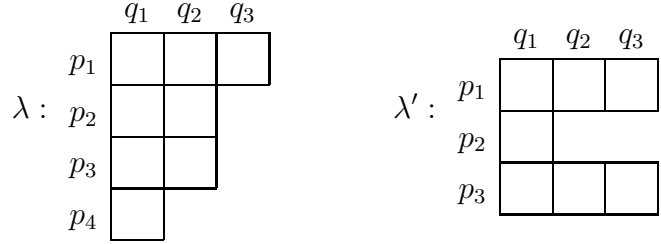
\[
\lam:
\def\mca#1{\multicolumn{1}{c}{#1}}
\def\mcb#1{\multicolumn{1}{c|}{#1}}
\renewcommand{\arraystretch}{1.4}
\begin{tabular}{c|c|c|c|}
  \mca{}  & \mca{$q_1$} & \mca{$q_2$} & \mca{$q_3$} \\\cline{2-4}
  \mcb{$p_1$}   &   &     &     \\\cline{2-4}
  \mcb{$p_2$}   &     &     & \multicolumn{1}{c}{}   \\\cline{2-3}
  \mcb{$p_3$} &     &     & \multicolumn{1}{c}{} \\\cline{2-3}
    \mcb{$p_4$} &     &   \multicolumn{1}{c}{}  & \multicolumn{1}{c}{}\\\cline{2-2}
\end{tabular}\qquad\qquad
\lam':
\def\mca#1{\multicolumn{1}{c}{#1}}
\def\mcb#1{\multicolumn{1}{c|}{#1}}
\renewcommand{\arraystretch}{1.4}
\begin{tabular}{c|c|c|c|}
  \mca{}  & \mca{$q_1$} & \mca{$q_2$} & \mca{$q_3$} \\\cline{2-4}
  \mcb{$p_1$}   &   &     &     \\\cline{2-4}
  \mcb{$p_2$}   &     &  \multicolumn{1}{c}{}   & \multicolumn{1}{c}{}   \\\cline{2-4}
  \mcb{$p_3$} &     &     & \\\cline{2-4}
\end{tabular}
\]

\begin{Lem}\label{lem:traspose-unipotent}\
\begin{enumerate}[\normalfont(1)]
\item
For $1\leq l\leq m$ and $1\leq k\leq n$,
\[
p_1+\cdots+p_l\leq l\cdot k+ q_{k+1}+\cdots+q_n.
\]
\item If furthermore, $p_1\geq \cdots \geq p_m >0$ and $k=p_{l+1}$, then equality holds in (1).
\end{enumerate}
\end{Lem}

\begin{proof}
(1)\ For $1\leq l\leq m$ and $1\leq k\leq n$,
\[
\begin{aligned}
&p_1+\cdots+p_l\\
\leq & \sum_{i=1}^n \min(q_i,l) \\
=& \sum_{i=1}^k \min(q_i,l)+\sum_{i={k+1}}^n \min(q_i,l)\\
\leq & l\cdot k+ q_{k+1}+\cdots+q_n.
\end{aligned}
\]
(2)\ If $p_1\geq \cdots \geq p_m >0$, then $q_i=\max\{j:p_j\geq i\}$ and
\[
p_1+\cdots+p_l
=  \sum_{i=1}^n \min(q_i,l).
\]
If $k=p_{l+1}=\max\{j: q_j\geq l+1\}$, then
\[
\min(q_i,l)=\left\{
\begin{aligned}
&l&\text{ if }1\leq i\leq p_{l+1},\\
& q_i &\text{ if }i\geq p_{l+1}+1.
\end{aligned}
\right.
\]
Thus the equality holds.
\end{proof}

\subsection{The action on simple roots}

Let us start with the following observations for general $w\in W(G)$. We represent $w$ by using the descending decomposition in Proposition \ref{prop:descending-decomposition}.

\begin{Lem}\label{lem:root-calculation}
If $i\leq j$, then
\[
s_j s_{j-1}\cdots s_i (\al_k)=\left\{
\begin{aligned}
&\al_k & \text{ if }k\leq i-2 \text{ or }k\geq j+2,\\
&\al_{i-1}+ \al_i+\cdots +\al_j &\text{ if }k=i-1,\\
&-(\al_i+\cdots +\al_j)&\text{ if }k=i,\\
&\al_{k-1} & \text{ if }j\geq k\geq i+1,\\
&\al_j+\al_{j+1} &\text{ if }k=j+1.\\
\end{aligned}
\right.
\]
\end{Lem}

\begin{proof}
A straightforward calculation.
\end{proof}

\begin{Lem}\label{lem:negative-root}
Suppose $i\leq j$. If $\pi_{k_j}(\al_l)$ is a negative root, so is $\pi_{k_i}\cdots\pi_{k_j}(\al_l)$.
\end{Lem}

        \begin{proof}
        If $\pi_{k_j}(\al_l)$ is negative, then the coefficient of $\al_j$ in $\pi_{k_j}(\al_l)$ is $-1$. Notice that the element $\pi_{k_i}\cdots \pi_{k_{j-1}}\in \la s_1,\cdots, s_{j-1}\ra$.  Thus the coefficient of $\al_j$ is $\pi_{k_i}\cdots \pi_{k_j} (\al_l)$ is also $-1$. This implies that $\pi_{k_i}\cdots \pi_{k_j}(\al_l)<0$.
        \end{proof}

\begin{Lem}\label{lem:positive-root}
Suppose $i\leq j$. If $\pi_{k_j}(\al_l)$ is  a positive root but not a simple root, then $\pi_{k_i}\cdots\pi_{k_j}(\al_l)>0$.
\end{Lem}

\begin{proof}
This follows by the same argument as in Lemma \ref{lem:negative-root}.
\end{proof}

\begin{Lem}\label{lem:root-less-than-cycle}
If $i\leq j$, then
\[
\#\{\al\in \Delta: \pi_{k_i}\cdots \pi_{k_j}(\al)<0 \}\leq j-i+1.
\]
\end{Lem}

\begin{proof}
We prove this by induction on the number of cycles. When $j-i+1=1$, the lemma follows from Lemma \ref{lem:root-calculation}.

Notice that
\[
\pi_{k_j}(\Delta)\subset \Delta \sqcup( \pi_{k_j}(\Delta)\cap \Phi^-) \sqcup(\pi_{k_j}(\Delta)\cap\Phi^+\bs \Delta).
\]
By Lemmas \ref{lem:root-calculation} and \ref{lem:negative-root},
$\pi_{k_i}\cdots\pi_{k_{j-1}}( \pi_{k_j}(\Delta)\cap \Phi^-)$ consists of at most one negative root, and  any element in $\pi_{k_i}\cdots\pi_{k_{j-1}}( \pi_{k_j}(\Delta)\cap \Phi^+\bs \Delta)$ is always positive.

Thus
\[
\begin{aligned}
&\pi_{k_i}\cdots \pi_{k_j}(\Delta)\cap \Phi^- \\
=&\pi_{k_i}\cdots\pi_{k_{j-1}}( \pi_{k_j}(\Delta))\cap \Phi^-\\
\subset &(\pi_{k_i}\cdots\pi_{k_{j-1}}( \Delta)\cap \Phi^- )\sqcup \pi_{k_i}\cdots\pi_{k_{j-1}}( \pi_{k_j}(\Delta)\cap \Phi^-).
\end{aligned}
\]
This finishes the induction step.
\end{proof}

\begin{Lem}\label{lem:descending-simple-roots}
 Suppose $i\leq j$. Consider $w=\pi_{k_i}\cdots\pi_{k_j}$ such that the sequence $k_i,\cdots, k_j$ is strictly increasing. For $k_j<l\leq j$, $w(\al_l)=\al_{l+i-j-1}$.
\end{Lem}

\begin{proof}
Notice that
\[
k_{j-1}< l-1, k_{j-2}<l-2, \cdots, k_i< l-j+i.
\]
Thus,
\[
w(\al_l)=\pi_{k_i}\cdots\pi_{k_j}(\al_l)=\pi_{k_i}\cdots\pi_{k_{j-1}}(\al_{l-1})=\cdots=\al_{l+i-j-1}.
\]
\end{proof}

\begin{Lem}\label{lem:simple-root-descending}
Consider $w=\Pi_2\cdots \Pi_a\in W(P)\bs W(G)$ and write $\Pi_a=\pi_{k_i}\cdots\pi_{k_j}$. Let $\al_m, \al_{m+1}, \cdots, \al_n$ be a sequence of $n-m+1$ simple roots. Then at least one of the following statements is true.
\begin{enumerate}[\normalfont(1)]
\item At least one of $w(\al_m), \cdots, w(\al_n)$ is positive.
\item  $\Pi_a(\al_m), \cdots, \Pi_a(\al_n)$ is still a sequence of consecutive simple roots.
\item   $w(\al_m)<0$ and $\Pi_a(\al_{m+1}), \cdots, \Pi_a(\al_n)$ is again a sequence of $n-m$ consecutive simple roots.
 \end{enumerate}
\end{Lem}

\begin{proof}
We prove this  by induction on the number of cycles in $\Pi_a$. If $j-i+1=1$, then one of the following holds:

 \begin{enumerate}[(a)]
  \item If $m+1\leq k_j\leq n+1$, then $\pi_{k_j}(\al_{k_j-1})>0$ but it is not a simple root. This implies $w(\al_{k_j-1})>0$.
  \item If $m-1\leq j\leq n-1$, then $\pi_{k_j}(\al_{j+1})>0$ but it is not a simple root. And similarly $w(\al_{j+1})>0$.
 \item If  $k_j\geq n+2$, or $m\geq j+2$, then $(\Pi_a(\al_m),\cdots,\Pi_a(\al_{n}))=(\al_m,\cdots,\al_n)$.
  \item  If $k_j=m$ and $j\geq n$, then $\Pi_a(\al_m)<0,$ and $(\Pi_a(\al_{m+1}),\cdots,\Pi_a(\al_{n}))=( \al_{m},\cdots, \al_{n-1})$.
 \item If $k_j+1\leq m\leq n\leq j$, then $(\Pi_a(\al_m),\cdots,\Pi_a(\al_{n}))=(\al_{m-1},\cdots, \al_{n-1})$.
 \end{enumerate}

Now we analyze the induction process. Write $\Pi_a=\pi_{k_i} \cdots \pi_{k_{j-1}}\cdot \pi_{k_j}$ and compute the action of $\pi_{k_j}$ first. There are $3$ cases.
\begin{enumerate}[(i)]
\item In cases (c) and (e), we are done by induction.
    \item In cases (a) and (b), we are done by Lemma \ref{lem:positive-root}.
\item   If we are in case (d), then by Lemma \ref{lem:negative-root}, we have $w(\al_m)<0$.  By Lemma \ref{lem:descending-simple-roots}, $\Pi_a(\al_{m+l})=\al_{m+l+i-j-1}.$
    Therefore, $\Pi_a(\al_{m+1}), \cdots, \Pi_a(\al_n)$ is still a sequence of $n-m$ consecutive simple roots.
\end{enumerate}
\end{proof}


 This immediately implies the following.

\begin{Lem}\label{lem:sequence-simple}
Let $w=\Pi_2\cdots \Pi_a\in W(P)\bs W(G)$. Let $\al_m, \al_{m+1}, \cdots, \al_n$ be a sequence of $n-m+1$ simple roots. If $n-m+1>a-1$, then there is some $m\leq l\leq n$ such that $w(\al_l)>0$.
\end{Lem}

\subsection{Proof of Theorem \ref{thm:root-theoretic} part (1)}

Recall we have a partition $\lam=(p_1\cdots p_m)$.  We fill in the Young diagram for $\lam$ with $\al_1,\cdots,\al_r$  from right to left in rows, from the top row  to the bottom row. Deleting the first column gives us $\Delta_\lambda$. For example, if $\lambda=(442)$ and $\lam^\top=(3322)$, then we have
\[
\def\mca#1{\multicolumn{1}{c}{#1}}
\def\mcb#1{\multicolumn{1}{c|}{#1}}
\renewcommand{\arraystretch}{1.5}
\begin{tabular}{c|c|c|c|c|}
  \mca{}      & \mca{$q_1$} & \mca{$q_2$} & \mca{$q_3$} & \mca{$q_4$} \\\cline{2-5}
  \mcb{$p_1$}   &   &   $\al_3$   & $\al_2$  & $\al_1$ \\\cline{2-5}
  \mcb{$p_2$}   & &  $\al_7$  &   $\al_6$  &  $\al_5$\\\cline{2-5}
  \mcb{$p_3$} &     &  $\al_9$   & \multicolumn{1}{c}{}& \multicolumn{1}{c}{} \\\cline{2-3}
\end{tabular}.
\]
If $\lambda=(413)$ and $\lam^\top=(3221)$, then we have
\[
\def\mca#1{\multicolumn{1}{c}{#1}}
\def\mcb#1{\multicolumn{1}{c|}{#1}}
\renewcommand{\arraystretch}{1.5}
\begin{tabular}{c|c|c|c|c|}
  \mca{}      & \mca{$q_1$} & \mca{$q_2$} & \mca{$q_3$} & \mca{$q_4$}\\\cline{2-5}
  \mcb{$p_1$}   &   & $\al_3$    &  $\al_2$  & $\al_1$ \\\cline{2-5}
  \mcb{$p_2$}   &     & \multicolumn{1}{c}{}  & \multicolumn{1}{c}{}  & \multicolumn{1}{c}{}\\ \cline{2-4}
  \mcb{$p_3$} &     &  $\al_7$  & $\al_6$ & \multicolumn{1}{c}{}\\ \cline{2-4}
\end{tabular}.
\]

For $1\leq i\leq m$, let $\Delta_{\lambda, i}$ denote the set of simple roots contained in the $i$th row. It is a sequence of simple roots with consecutive indices
and
\[
\Delta_{\lambda}=\bigsqcup_{1\leq i\leq m}\Delta_{\lambda,i}.
\]
For $2\leq i\leq n$, we also define $Q_i$ to be the set of simple roots contained in the $i$th column. Thus,
\[
\Delta_{\lam}=\bigsqcup_{2\leq i\leq n}Q_i.
\]

From now on suppose there exists $w\in W(P_{\mu^\top})\bs W(G)$ such that $w(\al)<0$ for all $\al\in\Delta_\lambda$.  We shall derive a contradiction. Write $w=\Pi_2\cdots \Pi_a$. Then for any $2\leq l\leq a$ and $\al\in\Delta_\lam$, $\Pi_l\cdots \Pi_a(\al)$ is either negative or a simple root (Lemma \ref{lem:positive-root}).  For each simple root $\al\in\Delta_\lam$, there is a smallest index $l$ such that $\Pi_{l+1}\cdots \Pi_a(\al)$ is a simple root but $\Pi_{l}\cdots \Pi_a(\al)<0$. We write $R_{l}$ for the set of such simple roots (notice that this depends on $w$). Thus
\[
\Delta_\lam=\bigsqcup_{2\leq l\leq a}R_l.
\]

From Lemma \ref{lem:simple-root-descending} we deduce the following result.
 \begin{Lem}\label{lem:key-obser}
With the above assumption,  if $\al_i\in R_l$ and $\al_{i+1}\in \Delta_\lam$, then $\al_{i+1}\in R_{2}\cup\cdots \cup R_{l-1}$; if $\al_i\in R_2$, then $\al_{i+1}\notin \Delta_\lam$.
\end{Lem}

Before we prove the most general result, we give three examples to illustrate the ideas of our proof.

\begin{Ex}
Let $\mu=(3^2)$, $\mu^\top=(2^3)$ and $\lam=(411)$. The corresponding Young tableaux are
\[
\mu: \qquad
\def\mca#1{\multicolumn{1}{c}{#1}}
\def\mcb#1{\multicolumn{1}{c|}{#1}}
\renewcommand{\arraystretch}{1.5}
\begin{tabular}{c|c|c|c|}
  \mca{}      & \mca{$r_1$} & \mca{$r_2$} & \mca{$r_3$}   \\\cline{2-4}
  \mcb{$t_1$}   &   &   $k_2$  &  $k_4$  \\\cline{2-4}
  \mcb{$t_2$}   &   & $k_3$  &   $k_5$  \\\cline{2-4}
\end{tabular}
\qquad\qquad
\lam:\qquad
\begin{tabular}{c|c|c|c|c|}
  \mca{}      & \mca{$q_1$} & \mca{$q_2$} & \mca{$q_3$} & \mca{$q_4$}\\\cline{2-5}
  \mcb{$p_1$}   &   &  $\al_3$    & $\al_2$  & $\al_1$ \\\cline{2-5}
  \mcb{$p_2$}   &     &   \multicolumn{1}{c}{}&  \multicolumn{1}{c}{}& \multicolumn{1}{c}{}\\\cline{2-2}
  \mcb{$p_3$} &     &   \multicolumn{1}{c}{} & \multicolumn{1}{c}{} & \multicolumn{1}{c}{}\\\cline{2-2}
\end{tabular}.
\]

In this case, $G=\GL_6$ and $P_{\mu^\top}\cong \GL_2\times \GL_2\times \GL_2$. Then $W(P_{\mu^\top})\bs W(G)$ consists of elements of the form
\[
(\pi_{k_2}\pi_{k_3}) (\pi_{k_4}\pi_{k_5}),
\]
where the cycles $(\pi_{k_2}, \pi_{k_3})$ and $(\pi_{k_4},\pi_{k_5})$ are both descending.  The result is clear from Lemma \ref{lem:sequence-simple}.

\end{Ex}

\begin{Ex} \label{ex:example2}
Let $\mu=(4221)$, $\mu^\top=(4311)$ and $\lambda=(3^3)$. Notice that
\[
p_1+p_2+p_3=3+3+3> 4+2+2=t_1+t_2+t_3.
\]
The associated Young tableaux are
\[
\mu: \qquad
\def\mca#1{\multicolumn{1}{c}{#1}}
\def\mcb#1{\multicolumn{1}{c|}{#1}}
\renewcommand{\arraystretch}{1.5}
\begin{tabular}{c|c|c|c|c|}
  \mca{}      & \mca{$r_1$} & \mca{$r_2$} & \mca{$r_3$} & \mca{$r_4$} \\\cline{2-5}
  \mcb{$t_1$}   &   &   $k_4$    & $k_7$ & $k_8$  \\\cline{2-5}
  \mcb{$t_2$} &     &  $k_5$ & \multicolumn{1}{c}{}& \multicolumn{1}{c}{} \\\cline{2-3}
  \mcb{$t_3$}   &     &  $k_6$  &  \multicolumn{1}{c}{}&  \multicolumn{1}{c}{} \\\cline{2-3}
  \mcb{$t_4$} &    &   \multicolumn{1}{c}{} & \multicolumn{1}{c}{}& \multicolumn{1}{c}{} \\\cline{2-2}
\end{tabular}
\qquad
\lam:\qquad
\begin{tabular}{c|c|c|c|}
  \mca{}      & \mca{$q_1$} & \mca{$q_2$} & \mca{$q_3$} \\\cline{2-4}
  \mcb{$p_1$}   &   &   $\al_2$   & $\al_1$  \\\cline{2-4}
  \mcb{$p_2$}   &   & $\al_5$ &  $\al_4$ \\\cline{2-4}
  \mcb{$p_3$} &    & $\al_8$   & $\al_7$ \\\cline{2-4}
\end{tabular}.
\]

 Notice $r_1=4$. Thus for any $w\in W(P_{\mu^\top})\bs W(G)$, there are at most $9-4=5$ cycles. However, $q_1=3$ and $\Delta_\lam$ has $9-3=6$ simple roots. By Lemma \ref{lem:root-less-than-cycle}, $w(\Delta_\lam)\cap \Phi^+\neq \emptyset$.
\end{Ex}

\begin{Ex}\label{ex:example3}
Let $\mu=(42222)$, $\mu^\top=(5511)$ and $\lambda=(333111)$. Notice that
\begin{equation}\label{eq:example3-1}
p_1+p_2+p_3=3+3+3> 4+2+2=t_1+t_2+t_2,
\end{equation}
but $r_1=5<6=q_1$. We cannot apply the argument in Example \ref{ex:example2}. The associated Young tableaux are
\[
\mu: \qquad
\def\mca#1{\multicolumn{1}{c}{#1}}
\def\mcb#1{\multicolumn{1}{c|}{#1}}
\renewcommand{\arraystretch}{1.5}
\begin{tabular}{c|c|c|c|c|}
  \mca{}      & \mca{$r_1$} & \mca{$r_2$} & \mca{$r_3$} & \mca{$r_4$}\\ \cline{2-5}
  \mcb{$t_1$}   &   &  $k_5$     & $k_{10}$ &  $k_{11}$  \\ \cline{2-5}
  \mcb{$t_2$} &     & $k_6$  & \multicolumn{1}{c}{}& \multicolumn{1}{c}{}  \\\cline{2-3}
  \mcb{$t_3$}   &    & $k_7$ &  \multicolumn{1}{c}{}&  \multicolumn{1}{c}{}\\\cline{2-3}
  \mcb{$t_4$} &     & $k_8$ & \multicolumn{1}{c}{}& \multicolumn{1}{c}{}  \\\cline{2-3}
    \mcb{$t_5$} &     & $k_9$ & \multicolumn{1}{c}{}& \multicolumn{1}{c}{}  \\\cline{2-3}
\end{tabular}
\qquad
\lam:\qquad
\begin{tabular}{c|c|c|c|}
  \mca{}      & \mca{$q_1$} & \mca{$q_2$} & \mca{$q_3$} \\\cline{2-4}
  \mcb{$p_1$}   &   &  $\al_2$    & $\al_1$ \\\cline{2-4}
  \mcb{$p_2$}   &    & $\al_5$ &  $\al_4$ \\\cline{2-4}
  \mcb{$p_3$} &    & $\al_8$   & $\al_7$ \\\cline{2-4}
  \mcb{$p_4$} &    & \multicolumn{1}{c}{}   & \multicolumn{1}{c}{} \\\cline{2-2}
    \mcb{$p_5$} &      & \multicolumn{1}{c}{}   & \multicolumn{1}{c}{} \\\cline{2-2}
      \mcb{$p_6$} &      & \multicolumn{1}{c}{}   & \multicolumn{1}{c}{} \\\cline{2-2}
\end{tabular}.
\]

Recall that
\[
R_2=\{\al\in\Delta_\lam: \Pi_3\Pi_4(\al)\in\Delta, \Pi_2\Pi_3\Pi_4(\al)<0\}.
\]
We claim that $R_2\subset Q_2=\{\al_2,\al_5,\al_8\}$. If not, then $R_2$ contains at least one of $\al_1,\al_4,\al_7$. Without loss of generality, we assume $\al_1\in R_2$. By Lemma \ref{lem:key-obser}, we deduce that $\al_2\notin \Delta_\lam$. This is a contradiction.

We now conclude that $R_2\subset Q_2$ and $R_3\sqcup R_4\supset Q_3$. By Lemma \ref{lem:root-less-than-cycle}, the number of cycles in $\Pi_3\Pi_4$ is greater than or equal to the size of $Q_3$. In other words,
\begin{equation}\label{eq:example3-2}
1+1=r_3+r_4\geq q_3=3.
\end{equation}
Contradiction.

We remark that
\[
r_3+r_4+3+3=t_1+t_2+t_3
\text{ and }
q_3+3+3=p_1+p_2+p_3.
\]
Thus Eq. (\ref{eq:example3-2}) actually contradicts Eq. (\ref{eq:example3-1}).
\end{Ex}

\begin{proof}[Proof of Theorem \ref{thm:root-theoretic} part (1)]

Now we turn to the general case.  There is an index $l$ such that $p_1+\cdots+p_l> t_1+\cdots +t_l$. Suppose there exists $w\in W(P_{\mu^\top})\bs W(G)$ such that $w(\Delta_\lambda)\subset \Phi^-$. We shall derive a contradiction. If $l=m$, then $r_1=b>m$. The number of cycles in $w\in W(P)\bs W(G)$ is at most $r+1-r_1$, but the number of simple roots in $\Delta_\lambda$ is $r-m$. This contradicts Lemma \ref{lem:root-less-than-cycle} (compare with Example \ref{ex:example2}).

If $l<m$, then
\begin{equation}\label{eq:contain}
R_{2}\sqcup\cdots \sqcup R_{t_{l+1}}\subset Q_2\sqcup\cdots \sqcup Q_{t_{l+1}}.
\end{equation}
Otherwise,  $R_{2}\sqcup\cdots \sqcup R_{t_{l+1}}$ contains a simple root $\al_k$ in $Q_{t_{l+1}+1}\sqcup\cdots \sqcup Q_{t_n}$.  Without loss of generality, we assume that $\al_k\in Q_{t_{l+1}+1}$. This means that
\[
\al_{k+1}\in Q_{t_{l+1}},\cdots, \al_{k+t_{l+1}-1}\in Q_2.
 \]
On the other hand, by Lemma \ref{lem:key-obser}, we deduce that
 \[
 \al_{k+1}\in R_{2}\sqcup\cdots \sqcup R_{t_{l+1}-1}, \cdots, \al_{k+t_{l+1}-2}\in R_2.
 \]
 This forces that $\al_{k+t_{l+1}-1}\notin \Delta_\lam$, which is impossible from our assumption. Thus Eq. (\ref{eq:contain}) holds and
\[
Q_{t_{l+1}+1}\sqcup\cdots \sqcup Q_{n}\subset R_{t_{l+1}+1}\sqcup\cdots \sqcup R_{a}.
\]
By Lemma \ref{lem:root-less-than-cycle},
\[
q_{t_{l+1}+1}+\cdots+q_n\leq r_{t_{l+1}+1}+\cdots +r_a.
\]

However, by Lemma \ref{lem:traspose-unipotent},
\[
\begin{aligned}
&l\cdot t_{l+1}+q_{t_{l+1}+1}+\cdots+q_n\\
&\geq p_1+\cdots+p_l\\
&> t_1+\cdots+t_l\\
&= l\cdot t_{l+1}+r_{t_{l+1}+1}+\cdots+r_a.
\end{aligned}
\]
This is a contradiction (compare with Example \ref{ex:example3}).
\end{proof}

\subsection{Proof of Theorem \ref{thm:root-theoretic} part (2)}
We need to show that there exists a unique $w_\mu=\Pi_2\cdots \Pi_a\in W(P_{\mu^\top})\bs W(G)$ such that $w_\mu(\Delta_\mu)$ is contained in  $\Phi^--\Phi_{\mu^\top}^-$; and $w(\Delta_\mu)\cap \Phi^+\neq \emptyset$ for all $w\neq w_\mu$. Before giving the proof we give an example to explain the idea.

\begin{Ex}
Let $\mu=\lam=(433)$, $\mu^\top=(3331)$.
The associated Young tableaux are
\[
\mu: \qquad
\def\mca#1{\multicolumn{1}{c}{#1}}
\def\mcb#1{\multicolumn{1}{c|}{#1}}
\renewcommand{\arraystretch}{1.5}
\begin{tabular}{c|c|c|c|c|}
  \mca{}      & \mca{$r_1$} & \mca{$r_2$} & \mca{$r_3$} & \mca{$r_4$} \\\cline{2-5}
  \mcb{$t_1$}   &   &   $k_3$    & $k_6$ & $k_9$  \\\cline{2-5}
  \mcb{$t_2$} &     &  $k_4$ & $k_7$ &  \multicolumn{1}{c}{} \\\cline{2-4}
  \mcb{$t_3$}   &     &  $k_5$  & $k_8$ &  \multicolumn{1}{c}{} \\\cline{2-4}
\end{tabular}
\qquad
\lam: \qquad
\def\mca#1{\multicolumn{1}{c}{#1}}
\def\mcb#1{\multicolumn{1}{c|}{#1}}
\renewcommand{\arraystretch}{1.5}
\begin{tabular}{c|c|c|c|c|}
  \mca{}      & \mca{$q_1$} & \mca{$q_2$} & \mca{$q_3$} & \mca{$q_4$} \\\cline{2-5}
  \mcb{$p_1$}   &   &   $\al_3$    & $\al_2$ & $\al_1$  \\\cline{2-5}
  \mcb{$p_2$} &     &  $\al_6$ & $\al_5$ &  \multicolumn{1}{c}{}\\\cline{2-4}
  \mcb{$p_3$}   &     &  $\al_9$  & $\al_8$ &  \multicolumn{1}{c}{} \\\cline{2-4}
\end{tabular}
\]
We have
\[
R_2=Q_2=\{\al_3,\al_6,\al_9\}, \qquad R_3=Q_3=\{\al_2,\al_5, \al_8\}, \qquad R_4=Q_4=\{\al_1\}.
\]
This allows us to find the desired $w_\mu$ inductively. Notice that a negative root is in $\Phi^--\Phi_{\mu^\top}^-$ if and only if the coefficients of $\al_3,\al_6$ or $\al_9$ is $-1$.

By Lemma \ref{lem:root-calculation}, the only choice for $\Pi_4$ corresponds to $k_9=1$. It is easy to check that
 $s_{987654321}(\al_1)=-(\al_1+\cdots+\al_9)$, and therefore $w(\al_1)\in \Phi^--\Phi_{\mu^\top}^-$
 since the coefficient of $\al_9$ is $-1$.

 We then need to find $\Pi_2\Pi_3$ such that $\Pi_2\Pi_3(w)<0$ for all $w\in \Pi_4(Q_2\cup Q_3)$. Notice that this is equivalent to the same problem, but with a smaller rank. To be more precise, we may naturally view $\Pi_2\Pi_3$ as an element in $W(P_{\mu'^\top})\bs W(\GL_9)$, where $\mu'=(333)^\top=(333)$. (In terms of the Young diagram, we only need to delete the last column.) On the other hand, $\Pi_4(Q_2\cup Q_3)=\{\al_1,\al_2,\al_4,\al_5,\al_7,\al_8\}$, which are exactly the simple roots we need to consider in this smaller rank case.

\[
\mu': \qquad
\def\mca#1{\multicolumn{1}{c}{#1}}
\def\mcb#1{\multicolumn{1}{c|}{#1}}
\renewcommand{\arraystretch}{1.5}
\begin{tabular}{c|c|c|c|}
  \mca{}      & \mca{$r_1$} & \mca{$r_2$} & \mca{$r_3$} \\\cline{2-4}
  \mcb{$t_1$}   &   &   $k_3$    & $k_6$   \\\cline{2-4}
  \mcb{$t_2$} &     &  $k_4$ & $k_7$ \\\cline{2-4}
  \mcb{$t_3$}   &     &  $k_5$  & $k_8$ \\\cline{2-4}
\end{tabular}
\qquad
\lam': \qquad
\def\mca#1{\multicolumn{1}{c}{#1}}
\def\mcb#1{\multicolumn{1}{c|}{#1}}
\renewcommand{\arraystretch}{1.5}
\begin{tabular}{c|c|c|c|}
  \mca{}      & \mca{$q_1$} & \mca{$q_2$} & \mca{$q_3$}\\\cline{2-4}
  \mcb{$p_1$}   &   &   $\al_2$    & $\al_1$  \\\cline{2-4}
  \mcb{$p_2$} &     &  $\al_5$ & $\al_4$ \\\cline{2-4}
  \mcb{$p_3$}   &     &  $\al_8$  & $\al_7$  \\\cline{2-4}
\end{tabular}
\]

Now we claim that there is only one choice for $\Pi_3$, corresponding to $(k_6, k_7, k_8)=(1,4,7)$. Recall that we need to find $\Pi_3$ which sends $\al_1,\al_4$, and $\al_7$ to negative roots. Indeed, if $k_8<7$, then $\pi_{k_8}(\al_7)=\al_6$ and by Lemma \ref{lem:descending-simple-roots}, $\Pi_3(\al_7)=\al_4$, which is positive; if $k_8=8$, then $w(\al_7)>0$; if $k_8>8$, then $\pi_{k_8}(\al_1,\al_4,\al_7)=\al_1,\al_4,\al_7$, and $\pi_{k_6}\pi_{k_7}$ cannot send $\al_1,\al_4,\al_7$ to negative roots.
Otherwise, this contradicts Lemma \ref{lem:root-less-than-cycle} and $Q_3=R_3$. Similarly we deduce that $k_7=4$ and  $k_6=1$. One may also notice that the coefficient of $\al_6$ in $\Pi_3(\al_1), \Pi_3(\al_4)$, and $\Pi_3(\al_7)$ are all $-1$, thus the coefficients of $\al_6$ in $w(\al_1), w(\al_4)$, and $w(\al_7)$ are also $-1$. This shows that they are in $\Phi^--\Phi^-_{\mu^\top}$.

Lastly, we need to choose $\Pi_2$ such that $\Pi_2(w)<0$ for all $w\in \Pi_3\Pi_4(Q_2)$. We may similarly see that this is equivalent to the problem with $\mu''^\top=(33)$. The corresponding Young tableaux are obtained by deleting the last row from $\mu'$ and we see that $(k_3,k_4,k_5)=(1,3,5)$. In other words, the unique choice for $w_\mu$ is
\[
(s_{321}s_{43}s_{5})(s_{654321}s_{7654}s_{87})(s_{987654321}).
\]

\end{Ex}
\begin{proof}[Proof of Theorem \ref{thm:root-theoretic} part (2) ]
Suppose $w(\al)<0$ for all $\al\in\Delta_{\mu}$. We prove the result by induction on $a$.  Clearly
$R_i=Q_i$ for $2\leq i\leq a$.

We first show that there is a unique choice for $\Pi_a$. Let $x=r_1+\cdots +r_{a-1}-1$. Then $(k_{x+1},\cdots, k_{x+r_a})$ are the entries in the last column of $\mu$. Let $(\al_{l_1},\cdots, \al_{l_{r_a}})$ be the simple roots in the last column of $\lam$. Indeed, one can check that $l_i=1+(i-1)p_1$.  We claim that
 \[
 (k_{x+1},\cdots, k_{x+r_a})=(l_1,\cdots, l_{r_a}).
 \]

 This is clear if $r_a=1$. Now we argue by induction on $r_a$. If $k_{x+r_a}<l_{r_a}$, then by Lemma \ref{lem:descending-simple-roots} $\Pi_a(\al_{l_{r_a}})=\al_{l_{r_a}-r_a}$, which is positive; if $k_{x+r_a}=l_{r_a}+1$, then $w(\al_{r_a})>0$; if $k_{x+r_a}>l_{r_a}+1$, then $(\al_{l_1},\cdots, \al_{l_{r_a}})$ are invariant under $\pi_{k_{x+r_a}}$, and $\pi_{k_{x+1}}\cdots \pi_{k_{x+r_a-1}}$ cannot send $(\al_{l_1},\cdots, \al_{l_{r_a}})$ to negative roots. Thus, $k_{x+r_a}=l_{r_a}$. Notice that for this choice, $(\al_{l_1},\cdots, \al_{l_{r_{a-1}}})$ are invariant under $\pi_{k_{x+r_a}}$. This allows us to define $\Pi_a$ inductively.

The above argument also shows that
$\Pi_a$ is unique. Moreover, it is easy to check that the coefficients of $\al_{x+1}$ in $\Pi_a(\al_{l_i})$'s are $-1$, thus the coefficients in $w(\al_{l_i})$'s are also $-1$. This implies that $w(\al_{l_i})\in \Phi^--\Phi^-_{\mu^\top}$.


Now notice that, $\Pi_2\cdots \Pi_{a-1}$ may be viewed as an element in $W(P_{\mu'^\top})\bs W(\GL_{r+1-r_a})$, where $\mu'=(r_1\cdots r_{a-1})^\top$. The operation on the Young diagram is simply deleting the last column. On the other hand, under the action of $\Pi_a$, the simple roots in the last column  are sent to negative roots. For $1\leq i\leq r_a$, the indices in the $i$th row drop by $i$; and the indices in the remaining rows drop by $r_a$. Thus we reduce the problem to $\GL_{r+1-r_a}$ with partition $(r_1\cdots r_{a-1})^\top$, which is of smaller rank. This allows us to define $w=\Pi_2\cdots \Pi_a$ inductively and uniquely.
\end{proof}

\section{Degenerate Eisenstein Series}\label{sec:degenerate-eis-series}
Let $F$ be a number field. Let $\ba$ be its adele ring. Let $\mu=(t_1\cdots t_b)$ be a partition of $r+1$  with $t_1\geq \cdots \geq t_b>0$. Let $\mu^\top=(r_1\cdots r_a)$ denote the transpose of $\mu$. Let $P=P_{\mu^\top}$.
 Let $\delta_P$ be the modular quasicharacter of $G$ with respect to $P$. We denote by $E_\mu(g,\underline{s})$ the Eisenstein series which corresponds to the induced representation $\Ind^{G(\ba)}_{P(\ba)}\delta_{P}^{\underline{s}}$ (this depends on a choice of test vectors, but we suppress this from the notation). Here $\underline{s}=(s_1,\cdots, s_{a})$ denotes a multi-complex variable.

We now define semi-Whittaker coefficients. Let $\lambda=(p_1\cdots p_m)$ be a general partition of $r+1$.
Fix a nontrivial additive character $\psi:F\bs\ba\to \bc^\times$. Let $\psi_\lambda:U(F)\bs U(\ba)\to \bc^\times$ be the  character  such that it acts as $\psi$ on the  root subgroups associated with $\al\in\Delta_\lam$, and acts trivially otherwise.
Given an automorphic form $f$ on $\GL_{r+1}(\ba)$, we define a $\lambda$-semi-Whittaker coefficient of $f$ as the integral
\[
\int\limits_{U(F)\bs U(\ba)} f(ug)\psi_\lambda(u) \ du.
\]

\begin{Thm}\label{thm:semi-whittaker}
Suppose $\mathrm{Re}(s_i)\gg0$.
\begin{enumerate}[\normalfont(1)]
\item If there is an index $l$ such that $p_1+\cdots +p_l> t_1+\cdots + t_l$, then
\[
\int\limits_{U(F)\bs U(\ba)} E_\mu(ug,\underline{s})\psi_\lambda(u) \ du= 0
\]
for all choices of data.
\item The semi-Whittaker coefficient
\[
\int\limits_{U(F)\bs U(\ba)} E_\mu(ug,\underline{s})\psi_\mu(u) \ du
\]
is nonzero for some choice of data.
\end{enumerate}
\end{Thm}

\begin{proof}
For $\mathrm{Re}(s_i)\gg0$ we unfold the Eisenstein series. Thus we need to study the space $P\bs G/U$, and analyze the contribution from each representative. By the Bruhat decomposition, we identify $P\bs G/U$ with $W(P)\bs W(G)$. To prove part (1), it suffices to show that for every $w\in W(P)\bs W(G)$, there is a $u\in U$ such that $\psi_\lam(u)\neq 1$ and $wuw^{-1}\in U$. This follows from Theorem \ref{thm:root-theoretic} part (1).

To prove part (2), again by Theorem \ref{thm:root-theoretic} part (2), we see that the only contribution comes from $w_\mu$. Thus,
\[
\int\limits_{U(F)\bs U(\ba)} E_\mu(ug,\underline{s})\psi_\lambda(u) \ du=\int\limits_{U_{w_\mu}(\ba)} f(w_\mu ug,\underline{s})\psi_\lambda(u) \ du,
\]
where $U_{w_\mu}$ is the subgroup of $U$ which corresponds to the roots $\al>0$ such that $w_\mu(\al)<0$.
The right-hand side is factorizable, and its value is a ratio of zeta functions. For $\mathrm{Re}(s_i)\gg0$ it is nonzero.  This completes the proof.
\end{proof}

The corresponding local result also holds and can be proved similarly; see also \cite{MoeW}.
\begin{Thm}\label{thm:semi-whittaker-local}
Let $F$ be a non-Archimedean local field.
\begin{enumerate}[\normalfont(1)]
\item If there is an index $l$ such that $p_1+\cdots +p_l> t_1+\cdots + t_l$, then
\[
\dim\Hom_{U(F)}(\Ind_{P(F)}^{G(F)}\delta_P^{\underline{s}},\psi_\lam)=0.
\]
\item If $\lambda=\mu$, then
\[
\dim\Hom_{U(F)}(\Ind_{P(F)}^{G(F)}\delta_P^{\underline{s}},\psi_\mu)=1.
\]
\end{enumerate}
\end{Thm}

\section{Unipotent Orbits and Fourier Coefficients}\label{sec:unipotent-orbits}

\subsection{Fourier coefficients associated with unipotent orbits}\label{subsec:fourier-unipontent}

Given a unipotent orbit, we can associate a set of Fourier coefficients. General references for unipotent orbits are Carter \cite{Ca} and Collingwood and McGovern \cite{CM}. For the local version of this association see \cite{Moe,MoeW}. For global details see Jiang and Liu \cite{JL} and Ginzburg \cite{ginzburg2006certain,ginzburgaim}.

We work with the global setup. Let $F$ be a number field, and $\ba$ be its adele ring. Fix a nontrivial additive character $\psi:F\bs \ba\to\bc^\times$. The unipotent orbits of $\GL_r$ are parameterized by partitions of $r$. Let $\sco=(p_1\cdots p_k)$ with $p_1+\cdots+p_k=r$ be a unipotent orbit.  We shall always assume $p_1\geq p_2\geq \cdots\geq p_k>0$. To each $p_i$ we associate the diagonal matrix
                \[
                \mathrm{diag}(t^{p_i-1},t^{p_i-3},\cdots,t^{3-p_i},t^{1-p_i}).
                \]
                Combining all such diagonal matrices and arranging them  in decreasing order of the powers of $t$, we obtain a one-dimensional torus $h_{\sco}(t)$. For example, if $\sco=(3^21)$, then
                \[
                h_\sco(t)=\mathrm{diag}(t^2,t^2,1,1,1,t^{-2},t^{-2}).
                \]

                The one-dimensional torus $h_{\sco}(t)$ acts on $U$ by conjugation. Let $\alpha$ be a positive root and $x_\alpha(a)$ be the one-dimensional unipotent subgroup in $U$ corresponding to the root $\alpha$. There is a nonnegative integer $m$ such that
                                    \begin{equation}\label{Conjugation}
                                            h_{\sco}(t)x_\alpha(a)h_{\sco}(t)^{-1}=x_\alpha(t^ma).
                                    \end{equation}
                On the subgroups $x_\al(a)$ which correspond to negative roots $\alpha$, the torus $h_\sco(t)$ acts with non-positive powers.

                Given a nonnegative integer $l$,  we denote by $U_l(\sco)$ the subgroup of $U$ generated by all $x_\alpha(a)$ satisfying the Eq. (\ref{Conjugation}) with $m\geq l$. We are mainly interested in $U_l(\sco)$ where $l=1$ or $l=2$.

                Let
                \[
                M(\sco)=T\cdot \la x_{\pm\alpha}(a):h_{\sco}(t)x_\alpha(a)h_\sco(t)^{-1}=x_\alpha(a)\ra.
                \]
                The algebraic group $M(\sco)$ acts by conjugation on the abelian group $U_2(\sco)/U_3(\sco)$. If the ground field is algebraically closed, then under this action of $M(\sco)$ on the group $U_2(\sco)/U_3(\sco)$, there is an open orbit. Denote a representative of this orbit by $u_2$. It follows from the general theory that the connected component of the stabilizer of this orbit inside $M(\sco)$ is a reductive group. Denote by $Stab_\sco^0$ this connected component of the stabilizer of $u_2$.

                 The group $M(\sco)(F)$ acts on the group of all characters of $U_2(\sco)(F)\bs U_2(\sco)(\ba)$. Consider the subset of all characters such that over the algebraic closure, the connected component of the stabilizer inside $M(\sco)(F)$ is  equal to $Stab_\sco^0$. We denote such a character by $\psi_{U_2(\sco)}$. Given an automorphic function $\varphi(g)$ on $\GL_r(\ba)$,  the Fourier coefficients we want to consider are
                 \[
                 \int\limits_{U_2(\sco)(F)\bs U_2(\sco)(\ba)}\varphi(ug)\ \psi_{U_2(\sco)}(u) \ du.
                 \]
                 In this way, we associate with each unipotent orbit $\sco$ a set of Fourier coefficients. When the partition is $\sco=(r)$, the Fourier coefficients associated to $\sco$ are the Whittaker coefficients.

Let us recall the partial ordering defined on the set of unipotent orbits. Given $\sco_1=(p_1\cdots p_k)$ and $\sco_2=(q_1\cdots q_l)$, we say that $\sco_1\geq \sco_2$ if $p_1+\cdots+p_i\geq q_1+\cdots+q_i$ for all $1\leq i\leq l$. If $\sco_1$ is not greater than $\sco_2$ and $\sco_2$ is not greater than $\sco_1$, we say that $\sco_1$ and $\sco_2$ are not comparable.

                \begin{Def}\label{def:unipotentorbit}
                Let $\pi$ be an automorphic representation of $\GL_r(\ba)$. Let $\sco(\pi)$ denote the set of unipotent orbits of $\GL_r$ defined as follows. A unipotent orbit $\sco$ is in $\sco(\pi)$ if $\pi$ has a nonzero Fourier coefficient which is associated with the unipotent orbit $\sco$, and for all $\sco'>\sco$, $\pi$ has no nonzero Fourier coefficient associated with $\sco'$.
                \end{Def}

                Definition \ref{def:unipotentorbit} is described in the global setup. One may have a similar definition in the local context where Fourier coefficients are replaced by twisted Jacquet modules. We omit the details.

                \begin{Rem}
                It is expected that for any automorphic representation $\pi$, the set $\sco_G(\pi)$ is a singleton (see \cite{ginzburg2006certain} Conjecture 5.4). In this paper, the notation $\sco_G(\pi)=\mu$  means that the set $\sco_G(\pi)$ is a singleton, consisting of the orbit $\mu$ only.
                \end{Rem}

\subsection{Connection to Semi-Whittaker Coefficients}\label{sec:relations-between-coefficients}
There is a strong relation between semi-Whittaker coefficients and Fourier coefficients associated with unipotent orbits. We prove a global version and state a local version in this section. \textit{Notice: the results and notations are independent of the rest of this paper.}
\subsubsection{Statements}
Let $\sco=(p_1\cdots p_m)$ be a unipotent orbit for $\GL_r$. Let $\lam=(q_1\cdots q_n)$ be a general partition of $r$. Define a character $\psi_\lam: U(F)\bs U(\ba)\to \bc^\times$ as in the previous sections.

\begin{Prop}\label{prop:app-b-1}
Let $\pi$ be an automorphic representation on $\GL_r(\ba)$. The following are equivalent:
\begin{enumerate}[\normalfont(1)]
\item The unipotent orbit attached to $\pi$ is
$\sco(\pi)=\sco.$
\item If $q_1+\cdots +q_i\geq p_1+\cdots +p_i$ for some $i$, then
\[
\int\limits_{U(F)\bs U(\ba)} f(ug)\psi_\lam (u)\ du=0
\]
for all choices of data; and if $\lam=\sco$, then
\[
\int\limits_{U(F)\bs U(\ba)} f(ug)\psi_\lam (u)\ du
\]
is nonzero for some choice of data.
\end{enumerate}

\end{Prop}

The corresponding local version is also true.
\begin{Prop}\label{prop:app-b-2}
Let $F$ be a non-Archimedean local field.
Let $\pi$ be a smooth representation on $\GL_r(F)$. The following are equivalent:
\begin{enumerate}[\normalfont(1)]
\item The unipotent orbit attached to $\pi$ is
$\sco(\pi)=\sco.$
\item If $q_1+\cdots +q_i\geq p_1+\cdots +p_i$ for some $i$, then the twisted Jacquet module
$J_{U(F),\psi_\lam}(\pi)=0$;
and if $\lam=\sco$, then
$J_{U(F),\psi_\lam}(\pi)\neq 0.$
\end{enumerate}
Moreover, if $p_i$'s have the same parity, then
\[
J_{U(F),\psi_\sco}(\pi)\cong J_{U(F),\psi_{U_2(\sco)}}(\pi).
\]

\end{Prop}

\begin{Rem}
When $\sco=(n^ab)$, these results are proved in \cite{cai2016} Theorem 7.7 and 7.8.

\end{Rem}

\begin{Rem}
A local result of similar flavor can also be found in Gomez, Gourevitch and Sahi \cite{gomez2015generalized} Theorem E.
\end{Rem}

We present a proof of the global result in this section. The local case is similar and we omit the details. To save space, we write $[V]=V(F)\bs V(\ba)$ for a subgroup $V$ of $U$.

\subsubsection{Part (2) implies part (1)}

We first show that part (2) implies part (1). To prove the vanishing part, we need to show that any unipotent orbit which is greater than or not comparable with $\sco$ does not support any Fourier coefficient. We establish several lemmas.

Let $\sco'=(q_1\cdots q_n)$ be a unipotent orbit which is greater than or not comparable with $\sco$. Then there exists $i$ such that $q_1+\cdots +q_i >p_1+\cdots +p_i$. Let $m'\geq q_1+\cdots +q_i$. Let $\underline{\epsilon}=(\epsilon_{q_1+\cdots + q_i},\cdots, \epsilon_{m'})$ with $\epsilon_{q_1+\cdots + q_i},\cdots, \epsilon_{m'}\in F$. Let $V_{m'}$ be the unipotent radical with Levi $\GL_1^{m'} \times \GL_{r-m'}$. The semi-Whittaker character $\psi_{(q_1\cdots q_i)}$ may be also viewed as a character on $[V_{m'}]$.  Define characters $\psi_{(q_1\cdots q_i), \underline{\epsilon}}:[V_{m'}]\to \bc^\times$,
\[
\psi_{(q_1\cdots q_i), \underline{\epsilon}}(v)=\psi_{(q_1\cdots q_i)}(v) \cdot \psi\left(\sum_{j=q_1+\cdots+q_i}^{m'} \epsilon_{j}v_{j,j+1}\right).
\]
\begin{Lem}\label{lem:app-B-1}
The integral
\[
\int\limits_{[V_{m'}]}
f(ug) \psi_{(q_1\cdots q_i), \underline{\epsilon}}(u)\ du=0
\]
for all choices of data.
In particular,
\[
\int\limits_{[V_{q_1+\cdots +q_i}]} f(ug)\psi_{(q_1\cdots q_i)}(u) \ du=0
\]
for all choices of data.

\end{Lem}

\begin{proof}
We prove this by induction on $r-m'$. If $r-m'=0$, then the integral involved is just a semi-Whittaker coefficient. The result follows from part (2). Assume the result is true for $r-m'$. Then we prove it for $r-m'-1$ if $m'-1\geq q_1+\cdots +q_i$. Let $R_{m'}$ be the unipotent subgroup such that $u_{j,l}=0$ unless $j=m'$. Expand the integral along this unipotent subgroup. Both nontrivial and trivial orbits give zero by induction. This proves the result.

\end{proof}

Define $U'_{(q_1\cdots q_i),k}$ as the subgroup of $V_{q_1+\cdots+q_i}$ such that $u\in U'_{(q_1\cdots q_i),k}$ if $u_{j,l}=0$ for $j=q_1+\cdots +q_k,q_1+\cdots+q_{k+1},\cdots, q_1+\cdots +q_i$. We also define $U'_{(q_1\cdots q_i),i+1}=V_{q_1+\cdots +q_i}$. The character $\psi_{(q_1\cdots q_i)}$ is still a character on $U'_{(q_1\cdots q_i),k}$.

\begin{Lem}\label{lem:app-B-2}
For $1\leq k\leq i$, the integral
\begin{equation}\label{eq:app-B-2}
\int\limits_{[U'_{(q_1\cdots q_i),k}]}f(ug)\psi_{(q_1\cdots q_i)}(u)\ du=0
\end{equation}
for all choices of data.
\end{Lem}

\begin{proof}
We prove this by induction on $i-k$. When $i-k=0$, the group $U'_{(q_1\cdots q_i),k}$ is actually $V_{q_1+\cdots+q_i-1}$. Expand the integral (\ref{eq:app-B-2}) along $R_{q_1+\cdots+q_i}$. Both trivial and nontrivial orbits give zero by Lemma \ref{lem:app-B-1}. This establishes the case $i=k$. Assume the result is true for $i-(k+1)$. Then we prove it for $k$. Indeed, we may write
\[
\begin{aligned}
&\int\limits_{[U'_{(q_1\cdots q_i),k}]}f(ug)\psi_{(q_1\cdots q_i)}(u) \ du\\
=&\int\limits_{[U'_{(q_{k+1}\cdots q_i),i-k-1}]}\ \int\limits_{[V_{q_1+\cdots +q_k-1}]}f(uvg)\psi_{(q_1\cdots q_k)}(u) \ du \ \psi_{(q_{k+1}\cdots q_i)}(v)\ dv.
\end{aligned}
\]
Here, we view $U'_{(q_{k+1}\cdots q_i),i-k-1}$ as a subgroup of $U$ via $g\mapsto \mathrm{diag}(I_{q_1+\cdots+q_k},g)$.
Expand the inner integral along $R_{q_1+\cdots +q_k}$. The trivial orbit is zero by induction; the nontrivial orbit gives the Fourier coefficient for a larger partition $(q_1,\cdots, q_{k-1}, q_k+q_{k+1},\cdots, q_i)$, which is also zero by induction. Thus the integral (\ref{eq:app-B-2}) is zero.
\end{proof}

Now we recall a corollary of root exchange. The local version is proved in \cite{cai2016} Lemma 6.5. The global version can be proved analogously.
    \begin{Lem}\label{rootexchange2}
                Let $\pi$ be an automorphic representation on $\GL_r(\ba)$. Let $\sco=(p_1 1^{n-p_1})$. Then
                        \[
                        \int\limits_{[U_2(\sco)]}f(ug)\psi_{U_2(\sco)}(u) \ du =0
                        \]
                        for all choices of data if and only if
                        \[
                        \int\limits_{[V_{p_1-1}]}f(ug)\psi_{(p_1)} (u) \ du=0
                        \]
                        for all choices of data.
                \end{Lem}
By applying this lemma repeatedly, we obtain the following result.

\begin{Lem}\label{lem:root-exchange}
Let $\sco=(q_1\cdots q_i 1^{r-q_1-\cdots q_i})$. Then
   \[
                        \int\limits_{[U_2(\sco)]}f(ug)\psi_{U_2(\sco)}(u) \ du =0
                        \]
                        for all choices of data
                        if and only if
                        \[
                        \int\limits_{[U'_{(q_1\cdots q_i),1}]}f(ug)\psi_{(q_1\cdots q_i)} (u) \ du=0.
                        \]
                        for all choices of data.
\end{Lem}

Now we are ready to prove the general case. Let $\sco''=(q_1\cdots q_i 1^{r-q_1-\cdots -q_i})$. We then have
\begin{equation}\label{eq:app-B-3}
\int\limits_{[U_2(\sco')]}f(ug)\psi_{U_2(\sco')}(u)\ du=\int\limits_{X}\int\limits_{[U_2({\sco''})]}f(uxg) \psi_{(q_1\cdots q_i)} \ du \ \psi_X(x) \ d x
\end{equation}
where $X$ is some subgroup which we don't need to specify.  From Lemmas \ref{lem:app-B-2} and \ref{lem:root-exchange}, we see that the inner integral in Eq. (\ref{eq:app-B-3}) is zero.  This proves the vanishing part.

To prove the nonvanishing part, it suffices to show that
\begin{equation}\label{eq:app-B-4}
\int\limits_{[U'_{(p_1\cdots p_m),1}]} f(ug)\psi_{(p_1\cdots p_m)}\ du\neq 0
\end{equation}
for some choice of data.
Notice that the integral in Eq. (\ref{eq:app-B-4}) is
\[
\int\limits_{[U'_{(p_2\cdots p_k),1}]} \ \int\limits_{[U_{(p_1)}]} f(vug) \psi_{(p_1)}(v) \psi_{U_2(p_2\cdots p_k)}(u) \ dv \ du
\]
where $U'_{(p_2\cdots p_k),1}$ is viewed as a subgroup of $U$ via $g\mapsto \mathrm{diag}(I_{p_1},g)$.
Expand the inner integral along the subgroup $R_{p_1}$. The nontrivial orbit gives $0$ since that corresponds to the orbit $(p_1+p_2, p_3 \cdots p_k)$. Thus the above integral equals
\[
\int\limits_{[U'_{(p_1\cdots p_m),2}]} f(ug)\psi_{(p_1\cdots p_m)}(u)\ du.
\]
We now repeat this process by expanding along $R_{p_1+p_2},\cdots, R_{p_1+\cdots+p_{m-1}}$. This shows that the above integral is
\begin{equation}\label{eq:app-B-7}
\int\limits_{[U]} f(ug)\psi_{(p_1\cdots p_m)} (u)\ du.
\end{equation}
This is nonzero for some choice of data from the assumption on semi-Whittaker coefficients.

\subsubsection{Part (1) implies part (2)}

We first consider the vanishing part. Let $\lam=(q_1\cdots q_n)$ such that $q_1+\cdots +q_i>p_1+\cdots +p_i$ for some $i$. We rearrange $(q_1\cdots q_i)$ into non-increasing order to obtain $(q_1' \cdots q_i')$. Consider the unipotent orbit $\sco'=(q'_1\cdots q'_i 1\cdots 1)$. This unipotent orbit does not support any Fourier coefficient. By conjugating the integral with a suitable Weyl group element and applying Lemma \ref{rootexchange2}, we see that
\begin{equation}\label{eq:app-C-after-root-exchange}
\int\limits_{[U'_{(q_1\cdots q_i),1}]} f(ug)\psi_{(q_1\cdots q_i)}(u) \ du=0
\end{equation}
for all choices of data.

\begin{Lem}
For $1\leq k\leq i+1$, the integral
\[
\int\limits_{[U'_{(q_1\cdots q_i),k}]} f(ug)\psi_{(q_1\cdots q_i)}(u) \ du=0
\]
for all choices of data.
In particular,
\begin{equation}\label{eq:app-B-5}
\int\limits_{[V_{q_1+\cdots +q_i}]} f(ug) \psi_{(q_1\cdots q_i)}(u) \ du=0
\end{equation}
for all choices of data.
\end{Lem}

\begin{proof}
We prove this by induction on $k$. The case $k=1$ follows from Eq. (\ref{eq:app-C-after-root-exchange}). Suppose the result is true for $k-1$, and we prove it for $k$. Expand the integral
 \[
 \int\limits_{[U'_{(q_1\cdots q_i),k-1}]} f(ug)\psi_{(q_1\cdots q_i)}(u) \ du
 \]
 along the unipotent subgroup $R_{q_1+\cdots +q_{k-1}}$. The nontrivial orbit gives zero, since the coefficients correspond to a larger orbit and we can apply induction. Thus we only have the trivial orbit and
\[
\int\limits_{[U'_{(q_1\cdots q_i),k}]} f(ug)\psi_{(q_1\cdots q_i)}(u) \ du=\int\limits_{[U'_{(q_1\cdots q_i),k-1}]} f(ug)\psi_{(q_1\cdots q_i)}(u) \ du=0
\]
for all choices of data.
\end{proof}

To prove the vanishing statement in part (2), notice that
\begin{equation}\label{eq:app-B-6}
\int\limits_{[U]} f(ug) \psi_{\lam}(u)\ du
\end{equation}
contains Eq. (\ref{eq:app-B-5}) as an inner integral. Thus Eq. (\ref{eq:app-B-6}) is zero for all choices of data.

Finally, we prove the nonvanishing statement in part (2). This follows from the same argument as in the previous section. Once we have the vanishing result,  Eq. (\ref{eq:app-B-7}) is nonzero for some choice of data if and only if Eq. (\ref{eq:app-B-4}) is nonzero for some choice of data.
 This finishes the proof.

\subsection{Main result}
The main result in this paper is the following theorem.

\begin{Thm}\label{thm:main-thm}
With the notations in Section \ref{sec:degenerate-eis-series}, suppose that $\mathrm{Re}(s_i)\gg0$. We then have
\[
\sco_{G}(E_\mu(g,\underline{s}))=\mu.
\]
\end{Thm}

The theorem follows from Theorem \ref{thm:semi-whittaker} and Proposition \ref{prop:app-b-1}. This confirms Conjecture 5.1 in Ginzburg \cite{ginzburg2006certain}. We also remark that the nonvanishing part is also proved in Ginzburg \cite{ginzburg2006eulerian} Section 3, Proposition 1. This also implies that any unipotent orbit of general linear groups can occur as the unipotent orbit attached to a specific automorphic representation.

Notice that by Theorem \ref{thm:semi-whittaker-local} and Proposition \ref{prop:app-b-2}, a local analogue to Theorem \ref{thm:main-thm} is also true.

\begin{Thm}\label{thm:main-local}
Let $F$ be a non-Archimedean local field. Then
\[
\sco_G(\Ind_{P(F)}^{G(F)}\delta_P^{\underline{s}})=\mu.
\]
\end{Thm}




\bibliographystyle{siam}

\end{document}